\documentclass[12pt]{amsart}
\usepackage{amscd,amssymb,epsfig,mathtools,accentset}
\calclayout

\numberwithin{equation}{section} 

\newcommand{\ud}{\,d} 
 
\newcommand{\R}{\mathbb{R}}

\newcommand\cof{\operatorname{cof}}

\renewcommand{\div}{\operatorname{div}}

\newcommand{\tir}[1]{\ensuremath{\overline {#1}}} 
\newtheorem{thm}{Theorem}[section] 
 
\newtheorem{lemma}[thm]{Lemma}

\newtheorem{defn}[thm]{Definition} 
 
\newtheorem{rem}[thm]{Remark}

\def\whsq{\vbox to 5.8pt 
{\offinterlineskip\hrule 
\hbox to 5.8pt{\vrule height 
5.1pt\hss\vrule height 5.1pt}\hrule}}

\def\<{\langle} 
\def\>{\rangle} 
\def\PP{{\mathop{{\rm I}\kern-.2em{\rm P}}\nolimits}} 
\def\FF{{\mathop{{\rm I}\kern-.2em{\rm F}}\nolimits}}   
\def\ZZ{{\mathop{{\rm I}\kern-.2em{\rm Z}}\nolimits}} 
 
\newlength{\sidemargin} 
\begin{document}
\title[]{
Standard finite elements for the numerical resolution of the elliptic Monge-Amp\`ere equation: mixed methods}

\thanks{
The author was partially supported by NSF DMS grant No 1319640.}
\author{Gerard Awanou}
\address{Department of Mathematics, Statistics, and Computer Science, M/C 249.
University of Illinois at Chicago, 
Chicago, IL 60607-7045, USA}
\email{awanou@uic.edu}  
\urladdr{http://www.math.uic.edu/\~{}awanou}

\maketitle

\begin{abstract}
We prove a convergence result for a mixed finite element method for the Monge-Amp\`ere equation to its weak solution in the sense of Aleksandrov. The unknowns in the formulation are the scalar variable and the Hessian matrix. 
\end{abstract}

\section{Introduction}
Let $\Omega$ be a convex polygonal domain of $\R^d$ for $d =2, 3$ with boundary $\partial \Omega$. Let $f \in L^{\infty}(\Omega), g \in C(\partial \Omega)$ with $0 < c_0 \leq f \leq c_1$ for constants $c_0, c_1 >0$.  
We assume that $g$ can be extended to a function $\tilde{g} \in C(\tir{\Omega})$ which is  convex in $\Omega$.
We are interested in a mixed finite element method for the nonlinear elliptic Monge-Amp\`ere equation: find a continuous convex function $u$ such that
\begin{align} \label{m1}
\begin{split}
\det D^2 u & = f \, \text{in} \, \Omega\\
u & = g \, \text{on} \, \partial \Omega.
\end{split}
\end{align}
The expression $\det D^2 u$ should be interpreted in the sense of Aleksandrov. We review the notion of Aleksandrov solutions in section \ref{Aleks}. For $u \in C^2(\Omega)$, $\det D^2 u$ is the determinant of the Hessian matrix 
$D^2 u=\bigg( (\partial^2 u) / (\partial x_i \partial x_j)\bigg)_{i,j=1,\ldots, d} $. 

We consider a mixed formulation with unknowns the scalar variable $u$ and the Hessian $D^2 u$. The scalar variable and the components of the Hessian are approximated by Lagrange elements of degree $k \geq d$. 
The method considered in this paper was analyzed from different point of views in \cite{Neilan2013} and  \cite{AwanouLiMixed1} for smooth solutions of \eqref{m1} and for $k \geq 3$. 
The case of quadratic elements in two dimension was handled in \cite{Awanou-Quadratic}. 
But the convergence of the discretization for non smooth solutions was not understood. We present in this paper a theory which explains these results. 

Let $\delta >0$ be a small parameter and let $\widetilde{\Omega}$ be a convex polygonal subdomain of $\Omega$. Let $f_m, g_m \in C^{\infty}(\tir{\Omega})$ such that $0 < c_2 \leq f_m \leq c_3$, $f_m$ converges uniformly to $f$ on $\tir{\Omega}$  and $g_m$ converges uniformly to $\tilde{g}$ on $\tir{\Omega}$ \cite{MongeC1Alex}.
It is known, c.f. \cite{Awanou-Std04v4}  that the  Aleksandrov solution of
\begin{equation}
\det  D^2 u_m = f_m  \ \text{in} \, \Omega, \, u_m=g_m \, \text{on} \, \partial \Omega, 
\label{m1m}
\end{equation}
converges uniformly on compact subsets of $\Omega$ to the  Aleksandrov solution $u$ of \eqref{m1}. We choose $\tilde{m}$ such that $| f(x) - f_{\tilde{m}}(x)| < \delta$, $| \tilde{g}(x) - g_{\tilde{m}}(x)| < \delta$  and $| u(x) - u_{\tilde{m}}(x)| < \delta$ for all $x \in \widetilde{\Omega}$.

Consider the problem: find $\tilde{u} \in C(\tir{\widetilde{\Omega}})$ convex on $\widetilde{\Omega}$ such that 
\begin{equation}
\det  D^2 \tilde{u} = f_{\tilde{m}}  \ \text{in} \, \widetilde{\Omega}, \, \tilde{u}=u_{\tilde{m}} \, \text{on} \, \tir{\Omega} \setminus \widetilde{\Omega}.
\label{m2}
\end{equation}
We prove the uniform convergence on compact subsets of $\widetilde{\Omega}$, for the finite element equations on $\widetilde{\Omega}$, of the discrete scalar variable to the convex Aleksandrov solution of \eqref{m2}. 
By unicity of the Aleksandrov solution $u$ of \eqref{m1}, we have $\tilde{u}=u$ in $\widetilde{\Omega}$ and hence as $\widetilde{\Omega} \to \Omega$, $\tilde{u}  \to u$ uniformly on compact subsets of $\Omega$. 
We may thus choose $\widetilde{\Omega}$ close enough to $\Omega$ such that $g$ is a good approximation of the Dirichlet data $ \tilde{u}=u_{\tilde{m}}$. 
The boundary data $g$ is used as an approximation of the solution on $\Omega \setminus \widetilde{\Omega}$. The solution $u$ of \eqref{m1} can then be approximated within a prescribed accuracy by first choosing $\tilde{m}$,  $\widetilde{\Omega}$ and then $h$ sufficiently small. We emphasize that the solution $\tilde{u}$ of \eqref{m2} is not necessarily smooth.

For the implementation one can simply take $ f_{\tilde{m}}=f$ and $\widetilde{\Omega}=\Omega$. It is in that sense that the results of this paper explains the numerical results obtained in \cite{Lakkis11b,Neilan2013}. For that reason we do not reproduce them in this paper.


We present the notion of Aleksandrov solution of \eqref{m1} using an analytical definition, following \cite{Rauch77}. It is based on approximation by smooth functions. The techniques used in this paper were successfully implemented in the context of standard finite difference discretizations in \cite{Awanou-Std-fdv5} and in the context of standard finite element methods in  \cite{Awanou-Std04v4}. The general methodology consists in
\begin{enumerate}
\item[1-] Prove the convergence and local uniqueness of the solution of the discrete problem when \eqref{m1} has a smooth solution. Using the continuity of the eigenvalues of a matrix as a function of its entries, prove local uniqueness when the discrete problem has a solution which is piecewise strictly convex.
\item[2-] Verify that the numerical method is robust enough to handle the standard tests for non smooth solutions. We emphasize that for non smooth solutions, the implementation of the mixed method should be done as in \cite{Neilan2013,Lakkis11b}. 
\item[3-] Consider a sequence of functions $f_m, g_m \in C^{\infty}(\tir{\Omega})$ such that $0 < c_2 \leq f_m \leq c_3$ for constants $c_2, c_3 >0$, $f_m$ converges uniformly to $f$ on $\tir{\Omega}$  and $g_m$ converges uniformly to $\tilde{g}$ on $\tir{\Omega}$. Consider then the Monge-Amp\`ere equations (with solutions not necessarily smooth)
\begin{equation*} 
\det D^2 u_m =f_m \, \text{in} \, \Omega,  u_m =g_m \, \text{on} \, \partial \Omega.
\end{equation*}
It is known that, see \cite{Awanou-Std04v4} for details, $u_m$ converges to the Aleksandrov solution $u$ of \eqref{m1} uniformly on compact subsets of $\Omega$.
\item[4-] Consider  a sequence of smooth uniformly convex domains $\Omega_s$ increasing to $\Omega$ \cite{Blocki97}, with the property that $\widetilde{\Omega} \subset \Omega_s$ for all $s$, and the problems with smooth solutions \cite{Trudinger08}
\begin{align*} 
\begin{split}
\det D^2 u_{\tilde{m} s}  = f_{\tilde{m}} \, \text{in} \, \Omega_s, 
u_{\tilde{m} s}  = g_{\tilde{m}}  \, \text{on} \, \partial \Omega_s.
\end{split}
\end{align*}
Again  one proves that, see \cite{Awanou-Std04v4}, $u_{\tilde{m}s}$ converges uniformly on compact subsets of $\widetilde{\Omega}$ to $\tilde{u}$ as $s \to \infty$.
\item[5-] Establish that the discrete approximation $u_{ms,h}$ of the smooth function $u_{ms}$, on $\widetilde{\Omega}$ and with boundary data $u_{ms}$, converges uniformly to $u_{ms}$ on $\widetilde{\Omega}$ as $h \to 0$. This usually takes the form of an error estimate with constants depending on derivatives of $u_{ms}$.
\item[6-] 
Because $\widetilde{\Omega}$ is an interior domain, interior Schauder estimates allow to get a uniform bound on the derivatives of $u_{\tilde{m}s}$. In other words, $u_{\tilde{m}s,h}$ converges uniformly to $u_{\tilde{m} s}$ on compact subsets of $\widetilde{\Omega}$ at a rate which depends on $\widetilde{\Omega}$  but is independent of $s$. 

\item[7-] The local equicontinuity of piecewise convex functions allow to take a subsequence in $s$. This gives a piecewise convex finite element function $u_h$ which solves the finite element problem on $\widetilde{\Omega}$. The approximation $u_h$ is shown to converge uniformly on compact subsets of $\widetilde{\Omega}$ to the solution of \eqref{m2}. Local uniqueness of the solution is a consequence of the work done in step 1. See also Remark \ref{local-unique}.

\end{enumerate}
Given the above program, the main technical difficulties consist in completing steps 1 and 7. For steps 1 and 2, in this paper, we build on the work done in \cite{Neilan2013,Lakkis11b,AwanouLiMixed1,Awanou-Quadratic}. 
For each type of discretization we have considered, step 7 requires new ideas. 

Given the definition of Aleksandrov solution, it is natural to expect a discrete version of the comparison principle for a discretization. The lack of such a comparison principle  is related to the difficulty of proving stability of the discretization for smooth solutions without assuming a bound on a high order norm of the solution. For that reason, we introduced the theoretical computational domain $\widetilde{\Omega}$ and fix the parameter $\tilde{m}$ in the regularization of the data.

Our focus on standard discretizations is motivated by the need to allow the efficient tools developed for computational mathematics such as adaptive mesh refinements and multigrid methods to be transferred seamlessly to the context of the Monge-Amp\`ere equation. We expect that the general strategy of this paper can be adapted to the various 
mixed methods proposed in  \cite{Dean2003,Feng2009a,Neilan2013}. 
The main difficulty is an error analysis for smooth solutions of \eqref{m1} with piecewise linear finite element approximations. We consider these and some lower order elements in \cite{Awanou-Quadratic-mixed}. See Remark \ref{weak-s}.

Monge-Amp\`ere equations arise in several applications of increasing importance, e.g. optimal transportation and reflector design. In fact, in optimal transportation problems, the approach through Aleksandrov solutions is more natural as it allows to treat discontinuous right hand sides.

In \cite{Dean2003,GlowinskiICIAM07,Feng2009a,Lakkis11b}, it was suggested that the issue could be approached through the notion of viscosity solution of \eqref{m1}. Both notions of viscosity and Aleksandrov solutions for \eqref{m1} coincide for $f >0$ and continuous on $\tir{\Omega}$ \cite{Guti'errez2001}. The investigation of numerical methods for \eqref{m1} through the notion of viscosity solutions is still an active research area. 
A proven convergence method for \eqref{m1} through the notion of viscosity solution was obtained through monotone finite difference schemes, see for example \cite{Froese13}. For an approach through the notion of Aleksandrov solution for the two dimensional problem, we refer to \cite{Oliker1988}. The geometric approach taken in \cite{Oliker1988} is different from the approach taken in this paper. Other finite element discretizations have been proposed, e.g. \cite{Bohmer2008,Davydov12,Glowinski2014}.




We organize the paper as follows. In the second section we introduce some notation and preliminaries. 
 In the last section we prove our convergence result for non smooth solutions. 


\section{Notation and Preliminaries} \label{notation}

Let $O$ be an open subset of $\mathbb{R}^d, d=2,3$. We use the usual notation $L^p(O), 2 \leq p \leq \infty$ for the Lebesgue spaces and $H^s(O), 1 \leq s < \infty$ for the Sobolev spaces of elements of  $L^2(O)$ with weak derivatives of order less than or equal to $s$ in $L^2(O)$. We recall that $H_0^1(O)$ is the subset of $H^1(O)$
of elements with vanishing trace on $\partial O$. We also recall that $W^{s,\infty}(O)$ is the Sobolev space of functions with weak derivatives of order less than or equal to $s$ in $L^{\infty}(O)$.
For a given normed space $X$, we denote by $X^{d}$ the space of vector fields with components in $X$ and by  $X^{d \times d}$ the space of matrix fields with each component in $X$.

The norm in $X$ is denoted by $|| . ||_X$ and we omit the subscript $O$ and superscripts $d$ and $d \times d$ when it is clear from the context. 
The inner product in $L^2(O), L^2(O)^d$, and $L^2(O)^{d \times d}$ is denoted by $(,)$ and we use $\< , \>$ for the inner product on  $L^2(\partial O)$ and $L^2(\partial O)^d$. For inner products on subsets of $O$, we will simply append the subset notation.  We recall that for a matrix $A$, $A_{ij}$ denote its entries and the cofactor matrix of
$A$, denoted $\cof A$, is the matrix with entries $(\cof A)_{ij}=(-1)^{i+j} \det(A)_i^j$ where $\det(A)_i^j$ is the determinant of the matrix obtained from $A$ by deleting its $i$th row and its $j$th column. For two matrices $A=(A_{ij})$ and $B=(B_{ij})$,  $A: B=\sum_{i,j=1}^n A_{ij} B_{ij}$ denotes their Frobenius inner product. A quantity which is constant is simply denoted by $C$.

For a scalar function $v$, we denote by $D v$ its gradient vector and recall that $D^2 v$ denotes the Hessian matrix of second order derivatives.  
The divergence of a matrix field is understood as the vector obtained by taking the divergence of each row.

\subsection{Discrete variational problem}
We denote by  $\mathcal{T}_h$ a triangulation of $ \widetilde{\Omega}$ into simplices $K$ and assume that $\mathcal{T}_h$ is quasi-uniform.  Let $n$ denote the unit outward normal vector of $\partial \widetilde{\Omega}$. 
We will need the broken Sobolev norm
$$
||v||_{H^k(\mathcal{T}_h)} = \bigg( \sum_{K \cap \widetilde{\Omega}, K \in \mathcal{T}_h} ||v||^2_{H^k(K)}
\bigg)^{\frac{1}{2}}.
$$
Analogously, we define $||v||_{W^{1,\infty}(\mathcal{T}_h)} = \max_{K \cap \widetilde{\Omega}, K \in \mathcal{T}_h } ||v||_{W^{1,\infty}(K)}$.
We denote by $V_h( \widetilde{\Omega})$  the standard Lagrange finite element space of degree $k \geq d$ and denote by $\Sigma_h( \widetilde{\Omega})$ the space of symmetric matrix fields with components in the Lagrange finite element space of degree $k \geq d$. We make the abuse of notation of denoting by $V_h$ and $\Sigma_h$ respectively the spaces $V_h(\widetilde{\Omega})$ and $\Sigma_h(\widetilde{\Omega})$. Similarly $(,)$ and $\<,\>$ will be used to denote inner products for functions defined on $\widetilde{\Omega}$ and $\partial \widetilde{\Omega}$ respectively.

Let $I_h$ denote the  standard Lagrange interpolation operator from $C( \widetilde{\Omega})$, the space of continuous functions on $ \widetilde{\Omega}$, 
into the space $V_h$. We use as well the notation $I_h$ for the matrix version of the Lagrange interpolation operator mapping  $C( \widetilde{\Omega})^{d \times d}$ 
into $\Sigma_h$.

We consider the problem:
find $(u_h, \sigma_h) \in V_h \times \Sigma_h$ such that
\begin{align} \label{m11h}
\begin{split}
(\sigma_h,\tau) + (\div \tau, D u_h) - \< D u_h, \tau n \> & = 0, \forall \tau \in \Sigma_h\\
(\det \sigma_h, v) & = ( f_{\tilde{m}}, v), \, \forall v \in V_h \cap H_0^1(\widetilde{\Omega})\\
u_h & =u_{\tilde{m}} \, \text{on} \, \partial \widetilde{\Omega}.
\end{split}
\end{align}
By $u_{h}=u_{\tilde{m}}$ on $\partial \widetilde{\Omega}$, we mean that our approximations are discontinuous on the boundary and that $u_h$ coincides with $u_{\tilde{m}}$ at the Lagrange points on $\partial \widetilde{\Omega}$.
From a practical point of view, since $\widetilde{\Omega}$ is assumed sufficiently close to $\Omega$, one considers the analogue of Problem \eqref{m11h} on $\Omega$ with $u_h=g$ on $\partial \Omega$ and $f$ at the place of $f_{\tilde{m}}$.

\subsection{Properties of the Lagrange finite element spaces} \label{lagrange}
We recall some properties of the Lagrange finite element space of degree $k \geq 1$ that will be used in this paper. They can be found in  \cite{Brenner02}. We have

Interpolation error estimates.
\begin{align} \label{interpol}
\begin{split}
||v - I_h v||_{L^{\infty}} & \leq C h^{k+1-\frac{d}{2}} |v|_{H^{k+1}},  \forall v \in H^{k+1}(\Omega).
\end{split}
\end{align}

Inverse inequalities
\begin{align} 
||v||_{L^{\infty}} & \leq C h^{-\frac{d}{2}} ||v||_{L^2}, \forall v \in V_h \label{inverse0} \\
||v||_{H^1} & \leq C h^{-1} ||v||_{L^2}, \forall v \in V_h \label{inverse1} \\
||v||_{W^{1,\infty}(\mathcal{T}_h)} & \leq C h^{-\frac{d}{2}} ||v||_{H^1}, \forall v \in V_h. \label{inverse2}
\end{align}

Scaled trace inequality 
\begin{align} 
||v||_{L^2(\partial \widetilde{\Omega})} &\leq C h^{-\frac{1}{2}} ||v||_{L^2(\widetilde{\Omega})},\ \forall v \in V_h(\widetilde{\Omega}). \label{trace-inverse}
\end{align}

\subsection{Error analysis of the mixed method for smooth solutions} \label{error}
Let us assume that the unique convex solution $u$ of of \eqref{m1} is in $H^3(\Omega)$ and put $\sigma=D^2 u$. 
Then $u$ satisfies  the following mixed problem:
find $(u,\sigma) \in H^2(\Omega) \times H^1(\Omega)^{d \times d}$ such that
\begin{align} \label{m11}
\begin{split}
(\sigma,\tau) + (\div \tau, D u) - \< D u, \tau n \> & = 0, \forall \tau \in H^1(\Omega)^{d \times d} \\
(\det \sigma,v) & = (f,v), \forall v \in H_0^1(\Omega)\\
u & = g \, \text{on} \, \partial \Omega.
\end{split}
\end{align} 
It is proved in \cite{AwanouLiMixed1} that the above variational problem is well defined. We will assume without loss of generality that $h \leq 1$.
We define for $\rho >0$, 
$$\bar B_h(\rho)=\{(w_h, \eta_h) \in V_h\times \Sigma_h,\ \| w_h-I_hu\|_{H^1 }\leq \rho,\ \|\eta_h-I_h\sigma\|_{L^{2}}\leq h^{-1}\rho\}$$
\begin{align}
\begin{split}\label{zh}
Z_h & =\{ \, (w_h, \eta_h) \in V_h\times \Sigma_h,  w_h =u \,  \text{on} \, \partial \widetilde{\Omega}, \\
&  \qquad \quad  (\eta_h, \tau)+(\div \tau, Dw_h)-\<Dw_h, \tau\cdot n\>=0, \forall  \tau\in \Sigma_h
 \, \} \, \text{and}
 \end{split}
\end{align}
\begin{equation} \label{ball-def}
B_h(\rho)=\bar B_h(\rho)\cap Z_h.
\end{equation}
We recall that, \cite[Lemma 3.5]{AwanouLiMixed1}, the ball $B_h(\rho)\neq \emptyset$ for 
$\rho = C_0 h^{k}$ for a constant  $C_0 >0$.

It follows from the analysis in \cite{Neilan2013,AwanouLiMixed1,Awanou-Quadratic} that \eqref{m11h} is well-posed for $k \geq d$ and for $(u, \sigma) \in H^{k+3}(\Omega) \times H^{k+1}(\Omega)^{d \times d}$  we have the error estimates
\begin{align}
||u-u_h||_{H^1(\widetilde{\Omega})} & \leq C h^k \label{u-h1-err}\\
||\sigma-\sigma_h||_{L^2(\widetilde{\Omega})} & \leq C h^{k-1} \label{s-l2-err}.
\end{align}

The constant $C$ can be taken to be a constant multiple of $||u||_{H^{k+3}(\Omega)}$.

\subsection{Algebra with matrix fields}

We collect in the following lemma some properties of matrix fields, the proof of which can be found in \cite{AwanouLiMixed1,AwanouPseudo10}. 

\begin{lemma}
We have for two matrix fields $\eta$ and $\tau$
\begin{align} \label{mean-v}
\det \eta - \det \tau = \cof(r \eta + (1-r) \tau): (\eta - \tau), 
\end{align}
for some $r \in [0,1]$.

For $d=2$ and $d=3$, and two matrix fields $\eta$ and $\tau$
\begin{align}
||\cof (\eta):\tau||_{L^2} & \leq C ||\eta||_{L^{\infty}}^{d-1} ||\tau||_{L^2}. \label{cof-est}
\end{align}
\end{lemma}

\subsection{Continuity of the eigenvalues of a matrix as a function of its entries}

Let $\lambda_1(A)$ and $\lambda_2(A)$ denote the smallest and largest eigenvalues of the symmetric matrix $A$. We have

\begin{lemma}  \label{lem-1}
Assume that $u \in C^2(\tir{\Omega})$. Then there exists constants $r, R >0$ independent of $h$ and a constant $C_{conv} > 0$ independent of $h$ such that for all $v_h \in V_h$ with
$v_h=u$ on $\partial \widetilde{\Omega}$ and 
$$
||v_h-I_h u||_{H^1} < C_{conv} h^{2},
$$
we have
$$
r \leq \lambda_1( D^2 v_h(x)) \leq \lambda_2( v_h(x))  \leq M, \forall x \in K \cap \widetilde{\Omega}.
$$ 
\end{lemma}

\begin{proof}
The result is a consequence of the assumptions $0 < c_0 \leq f \leq c_1$ and the continuity of the eigenvalues of a matrix as a function of its entries. See for example \cite[Lemma 3.1]{Awanou-Std01}.
\end{proof}

\section{Convergence of the discretization to the Aleksandrov solution} \label{weak}

\subsection{The Aleksandrov solution} \label{Aleks}

We denote by $K(\Omega)$ the cone of convex functions on $\Omega$ and by $B(\Omega)$ the space of Borel measures on $\Omega$. Let $M$ denote the mapping
\begin{align*}
M & :  C^2(\Omega) \cap K(\Omega) \to B(\Omega) \\
 M[v](B) &= \int_B \det D^2 v(x) \ud x, \, \text{for a Borel set} \, B.
\end{align*}

We equip $K(\Omega)$ with the topology of compact convergence, i.e. for $v_m, v \in K(\Omega)$, $v_m$ converges to $v$ if and only if $v_m$ converges to $v$ uniformly on compact subsets of $\Omega$. We endow 
$B(\Omega)$ with the topology of weak convergence of measures. 

\begin{defn} A sequence $\mu_m$ of Borel measures converges weakly to a Borel measure $\mu$ if and only if
$$
\int_{\Omega} p(x)\ud \mu_m  \to \int_{\Omega} p(x) \ud \mu,
$$
for every continuous function $p$ with compact support in $\Omega$.
\end{defn}
If the measures $\mu_m$ have density $a_m$ with respect to the Lebesgue measure, and $\mu$ has density $a$ with respect to the Lebesgue measure, we have
\begin{defn}
Let $a_m, a \geq 0$ be given functions. The sequence $a_m$ converges weakly to  $a$ as measures if and only if 
$$\int_{\Omega } a_m p \ud x \to \int_{\Omega } a p \ud x, $$ 
for all continuous functions $p$ with compact support in $\Omega$. 
\end{defn}

It can be shown that the mapping $M$ extends uniquely to a continuous operator on $K(\Omega)$,  c.f. \cite[Proposition 3.1]{Rauch77}. Thus we have

\begin{lemma} \label{weak-s} Let $v_m$ be a sequence of convex functions in $\Omega$ such that $v_m \to v$ uniformly on compact subsets of $\Omega$. Then the associated Monge-Amp\`ere measures $M [v_m]$ tend to $M[v]$ weakly. 
\end{lemma}
We can now define the notion of Aleksandrov solution of \eqref{m1}.
\begin{defn}
A convex function $u \in C(\tir{\Omega})$ is an Aleksandrov solution of \eqref{m1} if only if $u=g$ on $\partial \Omega$ and $M[u]$ has density $f$ with respect to the Lebesgue measure.
\end{defn}
It is shown in \cite[Proposition 3.4]{Rauch77} that the extension of the mapping $M$ to $K(\Omega)$ coincides with the definition of Monge-Amp\`ere measure as curvature measure. 

\begin{thm} [Theorem 1.1 \cite{Hartenstine2006} ] \label{weak-cont}
Let $\Omega$ be a bounded convex domain of $\R^d$ and assume that $g$ can be extended to a function $\tilde{g} \in C(\tir{\Omega})$ which is  convex in $\Omega$. Then if $f \in L^1(\Omega)$ , \eqref{m1} has a unique convex Aleksandrov solution in $C(\tir{\Omega})$ which assumes the boundary condition in the classical sense.
\end{thm}

\subsection{Properties of convex functions}

We will use the following lemma

\begin{lemma} \label{Arzela}
A uniformly bounded sequence $u_j$ of convex functions on a convex domain $\Omega$ is locally uniformly equicontinuous and thus has a pointwise convergent subsequence.
\end{lemma}

\begin{proof}
For $p_j \in \partial u_j(x)$ and $x \in \Omega$, we have by \cite[Lemma 3.2.1]{Guti'errez2001}
$$
|p_j| \leq \frac{|u_j(x)| }{d(x,\partial \Omega)} \leq \frac{C}{d(x,\partial \Omega)},
$$
for a constant $C$ independent of $j$. Arguing as in the proof of \cite[Lemma 1.1.6]{Guti'errez2001}, it follows that the sequence $u_j$ is uniformly Lipschitz and hence equicontinuous on compact subsets of $\Omega$. By the Arzela-Ascoli theorem, \cite[p. 179]{Royden}, we conclude that the result holds.

\end{proof}

We also note

\begin{lemma}[\cite{Rockafellar70} Theorem 25.7] \label{unif-grad-conv}
Let $v$ be a convex function which is finite and differentiable on $\Omega$. Let $v_m$ be a sequence of convex functions finite and differentiable on $\Omega$ such that $v_m$ converges pointwise to $v$ on $\Omega$. Then
$$
\lim_{m \to \infty} D v_m(x) = D v(x), \forall x \in \Omega.
$$
Moreover, the mappings $D v_m$ converge to $D v$ uniformly on every closed bounded subset of $\Omega$.
\end{lemma}

\begin{rem} \label{unif-grad-conv2}
For a sequence of convex polynomials on a convex open set $C$, the above theorem holds for every compact subset of $\tir{C}$. The proofs are identical.
\end{rem}

\subsection{Smooth and polygonal exhaustions of the domain}

Let $\Omega_s$ denote a sequence of smooth uniformly convex domains increasing to $\Omega$, i.e. $\Omega_s \subset \Omega_{s+1} \subset\Omega$ and $d(\partial \Omega_s, \partial \Omega) \to 0$ as $s \to \infty$. Here $d(\partial \Omega_s, \partial \Omega)$ denotes the distance between $\partial \Omega_s$ and $\partial \Omega$. 
The existence of the sequence $\Omega_s$ follows for example from the approach in \cite{Blocki97}. Without loss of generality, we may assume that $\widetilde{\Omega} \subset \Omega_s$ for all $s$.

We recall that $f_m$ and $g_m$ are $C^{\infty}(\tir{\Omega})$ functions such that  $0 < c_2 \leq f_m \leq c_3, f_m \to f$ and $g_m \to \tilde{g}$ uniformly on $\tir{\Omega}$. 
It follows from  \cite{Caffarelli1984} that 
the problem
\begin{align} \label{m1-sub-m}
\begin{split}
\det D^2 u_{m s} & = f_m \, \text{in} \, \Omega_s \\
u_{m s} & = g_m  \, \text{on} \, \partial \Omega_s,
\end{split}
\end{align}
has a unique convex solution $u_{m s} \in C^{\infty}(\tir{\Omega}_s)$. It is known, see \cite{Awanou-Std04v4} for details,  that the sequence $u_{m s}$  converges uniformly  on compact subsets of $\Omega$ to the unique convex solution $u_m \in C(\tir{\Omega})$ of the problem \eqref{m11}.

Put $\sigma_{ms} = D^2 u_{ms}$. We consider the subdomain problem, analogue of \eqref{m11}: find $(u_{\tilde{m} s} ,\sigma_{\tilde{m} s} ) \in H^2(\widetilde{\Omega}) \times H^1(\widetilde{\Omega})^{d \times d}$ such that
\begin{align*} 
\begin{split}
(\sigma_{\tilde{m} s} ,\tau) + (\div \tau, D u_{\tilde{m} s} ) - \< D u_{\tilde{m} s} , \tau n \> & = 0, \forall \tau \in H^1(\widetilde{\Omega})^{d \times d} \\
(\det \sigma_{\tilde{m} s} ,v) & = (f_{\tilde{m}},v), \forall v \in H_0^1(\widetilde{\Omega})\\
u_{\tilde{m} s}  & = g_{\tilde{m}} \, \text{on} \, \partial \widetilde{\Omega}.
\end{split}
\end{align*} 
 
We now consider the analogue of \eqref{m11h}: find $(u_{\tilde{m} s,h}, \sigma_{\tilde{m} s,h}) \in V_{h} \times \Sigma_{h}$ such that
\begin{align} \label{m11hs}
\begin{split}
(\sigma_{\tilde{m} s,h},\tau) + (\div \tau, D u_{\tilde{m} s,h}) - \< D u_{\tilde{m} s,h}, \tau n \> & = 0, \forall \tau \in \Sigma_{h}\\
(\det \sigma_{\tilde{m} s,h}, v) & = ( f_{\tilde{m}}, v), \, \forall v \in V_{h} \cap H_0^1(\widetilde{\Omega})\\
u_{\tilde{m} s,h} & = u_{\tilde{m}} \, \text{on} \, \partial \widetilde{\Omega}.
\end{split}
\end{align}

It follows from the results of section \ref{error} and the discussion above, that \eqref{m11hs} has a solution which satisfies for $h$ sufficiently small
\begin{align}
||u_{\tilde{m}s}-u_{\tilde{m}s,h} ||_{H^1} & \leq C_{\tilde{m}s} h^k  \label{u-h1-errs}\\
||\sigma_{\tilde{m} s}-\sigma_{\tilde{m} s,h} ||_{L^2} & \leq C_{\tilde{m} s} h^{k-1},  \label{s-l2-errs}
\end{align}
for a constant $C_{\tilde{m}s}$. Moreover, by Lemma \ref{lem-1}, $u_{\tilde{m} s,h}$ is piecewise strictly convex.

\subsection{Convergence of the discretization}


We can now state our main  theorem

\begin{thm}
Under the assumptions set forth in the introduction, Problem \eqref{m11h} has a local solution $(u_h,\sigma_h)$ for $h$ sufficiently small, with $u_h$ a piecewise  convex function. Moreover $u_h$ converges uniformly on compact subsets of $\widetilde{\Omega}$ to the 
unique $C(\widetilde{\Omega})$ solution of \eqref{m2} which is convex on $\widetilde{\Omega}$.
\end{thm}

\begin{proof} It remains to complete steps 6 and 7 of the program outlined in the introduction. 

{\bf Part 1}: The existence of a limit $u_{h}$.

By \eqref{u-h1-errs} and the inverse estimate \eqref{inverse2}, we have on $\widetilde{\Omega}$
$$
||u_{ms}-u_{ms,h} ||_{W^{1,\infty}} \leq C_{ms} h^{-\frac{d}{2}} ||u_{ms}-u_{ms,h} ||_{H^1} \leq   C_{ms} h^{k-\frac{d}{2}},
$$
where $C_{ms}$ depends on $||u_{ms}||_{C^{k+1}(\widetilde{\Omega})}$.

On each compact subset $K$ of $\Omega$, we have by the interior Schauder estimates,  \cite[Theorem 4]{Dinew} and \cite{Awanou-Std04} for details,
\begin{equation} \label{int-Schauder}
||u_{\tilde{m} s}||_{C^2(K)} \leq C_{\tilde{m}},
\end{equation}
where the constant $C_{\tilde{m}}$ depends on $\tilde{m}, c_2$, $\widetilde{\Omega}$, $d(K, \partial \Omega)$, $f_{\tilde{m}}$ and $\max_{x \in \Omega} |u_{\tilde{m} s}(x)|$. Moreover, by a bootstrapping argument 
we have
\begin{equation} \label{int-Schauder2}
||u_{\tilde{m} s}||_{C^{k+3}(K)} \leq ||u_{\tilde{m} s}||_{C^{k+3}(\widetilde{\Omega})}  \leq C_{\tilde{m}},
\end{equation}
as well. 

We conclude that the sequence in $s$ of piecewise convex functions $u_{ \tilde{m} s,h}$ is uniformly bounded on compact subsets of $\widetilde{\Omega}$, and hence by Lemma \ref{Arzela} has a convergent subsequence also denoted by $u_{\tilde{m} s,h}$ which converges pointwise to a function $u_{h}$. 
The function $u_{h}$ is piecewise convex as the pointwise limit of piecewise convex functions and the convergence is uniform on compact subsets of $\widetilde{\Omega}$.


Next, we note that for a fixed $h$, $u_{\tilde{m} s,h}$ is a piecewise polynomial in the variable $x$ of fixed degree $k$ and convergence of polynomials is equivalent to convergence of their coefficients. Thus $u_{h}$ is a piecewise polynomial of degree $k$. Moreover, the continuity conditions on $u_{\tilde{m} s,h}$ are linear equations involving its coefficients. Thus $u_{h}$ has the same continuity property as $u_{\tilde{m} s,h}$. In other words $u_{h} \in V_h$.

{\bf Part 2}: The existence of a limit function $\sigma_h$ and equations solved by the pair $(u_{h}, \sigma_h)$.

We have for $\tau \in \Sigma_h$
\begin{align*}
 - (\div \tau, D u_{\tilde{m} s,h}) + \< D u_{\tilde{m} s,h}, \tau n \> & = - \sum_{K \in \mathcal{T}_h } \int_{K \cap \widetilde{\Omega}  } (D u_{\tilde{m} h,s}) \cdot \div \tau \ud x \\
 & \qquad \qquad +  \int_{K \cap \partial  \widetilde{\Omega}} (D u_{\tilde{m} h,s}) \cdot (\tau n ) \ud x.
\end{align*}

Recall that $u_{\tilde{m} s,h}$ is  convex on $K \cap  \widetilde{\Omega}$ and converges uniformly on compact subsets of $K \cap  \widetilde{\Omega}$ to $u_{h}$.  
In this part of the proof, we consider the continuous extension of $u_{\tilde{m} s,h}$ and $u_h$ up to the boundary $\partial  \widetilde{\Omega}$. We then obtain from Remark \ref{unif-grad-conv2}

\begin{align*}
 \int_{K \cap \widetilde{\Omega}} (D u_{\tilde{m} h,s}) \cdot \div \tau \ud x & \to  \int_{K \cap \widetilde{\Omega} } (D u_{h}) \cdot \div \tau \ud x 
 \end{align*}
 and
 \begin{align*}
 \int_{K \cap \partial \widetilde{\Omega} } (D u_{\tilde{m} h,s}) \cdot (\tau n_s) \ud x & \to \int_{K \cap \partial \widetilde{\Omega} } (D u_{h}) \cdot (\tau n) \ud x.
\end{align*}
We conclude that as $s \to \infty$
\begin{align} \label{Fh-tau}
- (\div \tau, D u_{\tilde{m} s,h}) + \< D u_{\tilde{m} s,h}, \tau n \> \to F_{h}(\tau) \coloneqq (\div \tau, D u_{h}) - \< D u_{h}, \tau n\>.
\end{align}
For $h$ fixed, we have by \eqref{inverse1} and \eqref{trace-inverse}
\begin{align*}
| F_{h}(\tau)| & \leq C ( ||\tau||_{H^1} ||D u_{h}||_{L^2} +  ||\tau||_{L^2(\partial \widetilde{\Omega}) } ||D u_{h}||_{L^2(\partial \widetilde{\Omega} )} ) \\
                         & \leq C ( h^{-1} ||D u_{h}||_{L^2} + h^{-\frac{1}{2}}  ||D u_{h}||_{L^2(\partial \widetilde{\Omega} )} )  ||\tau||_{L^2}.
\end{align*}
Thus $F_{h}$ is continuous on $\Sigma_h$ and by the Riesz representation theorem, there exists a unique $\sigma_{h} \in \Sigma_h$ such that
$$
 F_{h}(\tau) = (\sigma_{h},\tau).
$$ 
 In other words $(u_{h}, \sigma_{h} ) \in V_h \times \Sigma_h$ solves
 
 \begin{equation} \label{eq-part1}
 (\sigma_{h},\tau) + (\div \tau, D u_{h}) - \< D u_{h}, \tau n \>  = 0, \forall \tau \in \Sigma_{h}.
 \end{equation}
 
 The result can also be stated as follows:
 \begin{align} \label{weak-cvg-sigma}
( \sigma_{\tilde{m} s,h},\tau) \to (\sigma_h,\tau) \ \text{as} \ s \to \infty \ \text{and} \ \forall \tau \in \Sigma_h.
 \end{align}
 
 It remains to show that 
 \begin{equation} \label{eq-part2}
 (\det \sigma_{h},v ) = (f ,v), \forall v \in V_h \cap H_0^1(\widetilde{\Omega}).
 \end{equation}
 
 By the error estimate for smooth solutions \eqref{s-l2-errs} and the interior Schauder estimate \eqref{int-Schauder2}, we have
 $$
 ||\sigma_{\tilde{m} s,h}||_{L^2} \leq C,
 $$
 for a constant $C$ which depends on $\widetilde{\Omega}$ and $\tilde{m}$ but is independent of $s$. By assumption $\sigma_{\tilde{m} s,h} \in H^1(\widetilde{\Omega})$ and by 
\eqref{inverse1}
$$
||\sigma_{\tilde{m} s,h} ||_{H^1} \leq C h^{-1} ||\sigma_{\tilde{m} s,h} ||_{L^2} \leq C_{} h^{-1}.
$$
 Since $d=2,3$, by the compactness of the embedding of $H^1(\widetilde{\Omega})$ into $L^4(\widetilde{\Omega})$, up to a subsequence $\sigma_{\tilde{m} s,h}$ converges to a function $\hat{\sigma}_h$ in $L^4(\widetilde{\Omega})$.
 We thus have
 \begin{align} \label{weak-cvg-sigma2}
( \sigma_{\tilde{m} s,h},\tau) \to (\hat{\sigma}_h,\tau) \ \text{as} \ s \to \infty \ \text{and} \ \forall \tau \in \Sigma_h.
 \end{align}
By \eqref{weak-cvg-sigma} and \eqref{weak-cvg-sigma2} we have
\begin{equation} \label{rellich-proj}
(\hat{\sigma}_{h}, \tau) = (\sigma_{h},\tau) \ \forall \tau \in \Sigma_h, \ \text{i.e.} \  \sigma_h = P_{\Sigma_h}(\hat{\sigma}_h),
\end{equation}
 where we denote by $P_{\Sigma_h}$ the $L^2$ projection into $\Sigma_h$.
 
We have by \eqref{mean-v} and \eqref{rellich-proj} for some $r \in [0,1]$
\begin{align*}
\int_{\widetilde{\Omega}} (\det \sigma_{\tilde{m} s,h} - \det \sigma_{h}) v \ud x & = \int_{\widetilde{\Omega}}  \cof ( r \sigma_{\tilde{m} s,h} + (1-r) \sigma_{h} ):( \sigma_{\tilde{m} s,h} - \sigma_{h}) v \ud x \\
& = \int_{\widetilde{\Omega}} P_{\Sigma_h} \big( v \cof ( r \sigma_{\tilde{m} s,h} + (1-r) \sigma_{h} ) \big):( \sigma_{\tilde{m} s,h} - \sigma_{h})  \ud x \\
& = \int_{\widetilde{\Omega}} P_{\Sigma_h} \big( v \cof ( r \sigma_{\tilde{m} s,h} + (1-r) \sigma_{h} ) \big):( \sigma_{\tilde{m} s,h} - \hat{\sigma}_{h})  \ud x
\end{align*}
and thus
\begin{align*}
\bigg| \int_{\widetilde{\Omega}} (\det \sigma_{\tilde{m} s,h} - \det \sigma_{h}) v \ud x \bigg| & \leq C ||v||_{L^{\infty}}  || \cof ( r \sigma_{\tilde{m} s,h} + (1-r) \sigma_{h} )||_{L^{\infty}} || \sigma_{\tilde{m} s,h} - \hat{\sigma}_{h}||_{L^1} \\
& \leq C  ||v||_{L^{\infty}} (||\sigma_{\tilde{m} s,h}||_{L^{\infty}} + ||\sigma_{h}||_{L^{\infty} })^{d-1} || \sigma_{\tilde{m} s,h} - \hat{\sigma}_{h}||_{L^1} \\
& \leq C  ||v||_{L^{\infty}} (h^{-\frac{d}{2}} ||\sigma_{\tilde{m} s,h}||_{L^{2}} + h^{-\frac{d}{2}} ||\sigma_{h}||_{L^{2} })^{d-1} || \sigma_{\tilde{m} s,h} - \hat{\sigma}_{h}||_{L^4} \\
& \leq C  ||v||_{L^{\infty}} (C h^{-\frac{d}{2}} + h^{-\frac{d}{2}} ||\sigma_{h}||_{L^{2} })^{d-1} || \sigma_{\tilde{m} s,h} - \hat{\sigma}_{h}||_{L^4}\\
& \to 0 \ \text{as} \ s \to \infty.
\end{align*}
 
On the other hand, $(\det \sigma_{\tilde{m} s,h},v) = (f_{\tilde{m}}, v)$. It follows that  \eqref{eq-part2} holds, i.e.  since $u_h=u_{\tilde{m}}$ on $\partial \widetilde{\Omega}$, the pair $(u_h,\sigma_h)$ solves \eqref{m11h}. 

 {\bf Part 3}: Uniform convergence on compact subsets of $\widetilde{\Omega}$ of $u_h$ to $\tilde{u}$.
 
Since $u_{\tilde{m} s,h}$ is uniformly bounded on compact subsets of $\widetilde{\Omega}$, so is $u_h$. It follows from Lemma \ref{Arzela} that there exists a subsequence $u_{h_l}$ which converges pointwise to a piecewise convex function $v$. The latter is continuous on $\widetilde{\Omega}$ as it is locally finite. Moreover the convergence is uniform on compact subsets of $\widetilde{\Omega}$.

Let $K$ be a compact subset of $\widetilde{\Omega}$. There exists a subsequence $u_{\tilde{m}, s_l, h}$ which converges uniformly to $u_h$ on $K$.  By the uniform convergence of $u_{ms}$ to $u_m$,  $u_{\tilde{m}, s_l}$ converges uniformly to $\tilde{u}$ on $K$.

Let now $\epsilon >0$. 
Since $u_{h_l}$ converges uniformly on $K$ to $v$, $\exists l_0$ such that $\forall l \geq l_0$ $|u_{h_l}(x) - v(x)|< \epsilon/6$ for all $x \in K$.
 
There exists $l_1 \geq 0$ such that for all $l \geq \max \{ \,l_0,l_1 \, \}$, 
$|u_{\tilde{m}  s_l,h_l}(x) - u_{h_l}(x)|< \epsilon/6$ for all $x \in K$.

Moreover, there exists $l_2 \geq 0$ such that for all $l \geq \max \{ \,l_0,l_1, l_2 \, \}$, $|u_{\tilde{m}  s_l}(x) - \tilde{u}_{}(x)|< \epsilon/6$ for all $x \in K$.

By \eqref{u-h1-errs} we have on $K$  $|u_{\tilde{m}  s,h_l}(x) - u_{\tilde{m} s}(x)|\leq C h_l$ for all $x \in K$. We recall that the constant $C$ is independent of $s$ but depends on $\widetilde{\Omega}$ and $\tilde{m}$.

We conclude that for $ l \geq \max \{ \,l_0,l_1, l_2 \, \}$, $|\tilde{u}(x) - v(x) | <  \epsilon/2 + C h_l$ for all $x \in K$.
We therefore have for all $\epsilon >0 $ $|\tilde{u} (x) - v(x) | < \epsilon$. We conclude that $\tilde{u}=v$ on $K$. 

Since $u_h=u_{\tilde{m}}$ on $\partial \widetilde{\Omega}$ it follows that $v=\tilde{u}$ on $\partial \widetilde{\Omega}$. This proves that $\tilde{u}=v$ on $\widetilde{\Omega}$.

We conclude that $u_h$ converges uniformly on compact subsets of $\widetilde{\Omega}$ to $\tilde{u}$.

 The proof is complete.
 
\end{proof}

\begin{rem} \label{local-unique}
Let $x_0 \in \Omega$. We may assume that the solution $u_h$ is strictly convex by identifying $u_h$ with $u_h + \epsilon |x-x_0|^2$, where $| . |$ denotes the Euclidean norm in $\R^d$ and $\epsilon$ is taken to be close to machine precision. 
The arguments in \cite{AwanouLiMixed1} can then be repeated to show that the solution of \eqref{m11h} given in the previous theorem is locally unique. Using the notation of  \cite{Awanou-Quadratic} for the 2D problem, one only needs to take the rescaling parameter $\alpha$ equal to $\nu h^{k+2}$.
\end{rem}

\begin{rem} \label{local-unique2}
For $k \geq d+1$, a different proof of local uniqueness can be given based on the fixed point argument of \cite{Awanou-Quadratic} . Arguing as in \cite[section 7.1]{Awanou-Std04v4}, one shows that $\det \sigma_h >0$. We can then define a discrete strictly convex function as determined by a pair $(u_h, \sigma_h)$ for which $\sigma_h$ is a positive definite matrix. However, in \cite{Awanou-Quadratic}, to handle the quadratic case, we essentially used the positive definiteness of $D^2 u_h$ computed piecewise. Similar arguments can be given by using instead the positive definiteness of $\sigma_h$. Given the length of this paper, we wish to address this approach in a separate work.

\end{rem}

\begin{rem}
The only reason we assume that $k \geq 3$ in three dimension is that 
 the numerical solution in the case $k=2$ is much closer to the Lagrange
interpolant than what can be observed numerically using the approximation property of the
Lagrange finite element spaces. The solution is to use a rescaled version of the equation, i.e. $\det D^2 \beta u= \beta^d f$, for a suitable $\beta >0$. The same argument applies to the analysis in \cite{Awanou-Std01}. 
\end{rem}

\begin{rem} \label{about-weak-s}
One of the main tools used in \cite{Awanou-Std-fdv5} and \cite{Awanou-Std04v4} is the weak convergence of Monge-Amp\`ere measures, c.f. Lemma \ref{weak-s}. It requires the approximation of the scalar variable to be at least piecewise quadratic. In this paper, we were able to take advantage of the mixed formulation to prove weak convergence.
\end{rem}

\begin{rem}
If one assumes that the domain $\Omega$ is smooth and uniformly convex, one can take $\widetilde{\Omega} = \Omega$ and use global Schauder estimates, c.f. \cite{Awanou-Std04v4} for details. For the implementation, one should use Nitsche method to enforce the boundary condition and curvilinear coordinates for elements near the boundary. We refer to \cite{Brenner2010b} for this approach in the context of smooth solutions. One can also use isogeometric analysis as in \cite{Awanou-Iso}.
\end{rem}



\end{document}